\documentclass[12pt]{amsart}

\usepackage{
amsfonts,
latexsym,
amssymb,
}

\newcommand{\labbel}{\label}

\newtheorem{theorem}{Theorem}%[section]

\newtheorem{proposition}[theorem]{Proposition} 
 
\newtheorem{corollary}[theorem]{Corollary}

\newtheorem*{theorem*}{Theorem}
\newtheorem*{corollary*}{Corollary}

\theoremstyle{definition}

\theoremstyle{remark}

\DeclareMathOperator{\cf}{cf}

\newcommand{\m}{\mathfrak}
\begin{document}
 
\title{Productivity of $[\mu, \lambda ]$-compactness}

\author{Paolo Lipparini} 
\address{Dipartimento di Matematica\\ Viale delle Riviste Scientifiche\\II Universit\`a di Roma (Tor Vergata)\\I-00133 ROME ITALY}
\urladdr{http://www.mat.uniroma2.it/\textasciitilde lipparin}



\begin{abstract}
We show that 
if $\mu \leq \cf \lambda $ and 
$\lambda$ is  
a strong limit singular cardinal,
then
$[\mu, \lambda ]$-compactness is productive
if and only if either $\mu= \omega $, or  
$\mu$ is $\lambda$-compact.
\end{abstract} 
 
\maketitle  

If $\mu$ and $\lambda$  are infinite cardinals, a topological space is 
\emph{$[\mu, \lambda ]$-compact} if  every open cover by at most $\lambda$  sets has a subcover by less than $\mu$  sets.
We say that $[\mu, \lambda ]$-compactness is \emph{productive} 
if
every product 
of $[\mu, \lambda ]$-compact topological spaces 
is $[\mu, \lambda ]$-compact.
 By a well known theorem of 
 Stephenson  and  Vaughan \cite[Theorem 1.1]{SV}, if $\lambda$ is a strong limit singular cardinal, then 
 $[ \omega , \lambda ]$-compactness (that is, 
\emph{initial $\lambda$-compactness})
is productive. We generalize this result to the case 
when $\mu$ is a strongly compact cardinal, and show that some amount of 
strong compactness is indeed necessary to obtain productivity.

Among the many possible
definitions of strong compactness, we shall use the following one.
A cardinal $ \mu > \omega $   is said to be  \emph{$\lambda$-compact}
 if  there exists a $\mu$-complete $( \mu, \lambda )$-regular 
 ultrafilter. 
Recall that an ultrafilter $D$ is \emph{$(\mu, \lambda)$-regular} provided there exists a family of $ \lambda $-many members of $D$ every intersection of $\mu$-many of which is empty. See \cite{mru} for a survey on regularity of ultrafilters, and for applications to topology,
as well as to other fields.
The cardinal
$\mu$  is \emph{strongly compact} if and only if it is $\lambda$-compact for all 
cardinals $\lambda$.  It is well known that the above definitions are
equivalent to the more usual ones (see, e.~g., Kanamori and Magidor \cite[Section 15]{KM}). 
Notice that, in the above terminology, a cardinal $\mu$ 
is \emph{measurable} if and only if it is
$\mu$-compact. 

We shall state some results in a general form 
relative to a class of topological spaces.
We say that $[\mu, \lambda ]$-compactness is \emph{productive 
for members of a class $\mathcal K$}
if
every product 
$\prod _{j \in J} X_j $ 
of $[\mu, \lambda ]$-compact topological spaces belonging to $\mathcal K$
is $[\mu, \lambda ]$-compact.
Notice that we are not assuming that $\mathcal K$ is closed under products;
hence we are only asking that $\prod _{j \in J} X_j $ is $[\mu, \lambda ]$-compact, 
but it does not necessarily   belong to $\mathcal K$.
Notice also that, trivially, if $\mathcal K' \subseteq \mathcal K$, 
then 
productivity of 
$[\mu, \lambda ]$-compactness
for members of  $\mathcal K$
implies 
productivity of 
$[\mu, \lambda ]$-compactness
for members of  $\mathcal K'$.
In particular, for every class $\mathcal K$ of topological spaces,
productivity of 
$[\mu, \lambda ]$-compactness
implies productivity of 
$[\mu, \lambda ]$-compactness
for members of  $\mathcal K$.

For every infinite cardinal $\mu$, let $\mathcal C_ \mu$ be the class of all infinite regular cardinals $<\mu$, each one endowed with the order topology.

Recall that if $D$ is an ultrafilter over some set $I$, then a topological space $X$ is said to be \emph{$D$-compact}
  if every $I$-indexed sequence $(x_i) _{i \in I} $ of elements of $X$ has some \emph{$D$-limit point} in $X$,
that is, there is a  point $x \in X$ such that 
 $\{ i \in I \mid x_i \in U\} \in D$,
for every open neighborhood $U$ of $x$.

\begin{theorem} \labbel{thm}
Suppose that $ \omega <\mu \leq \lambda $ are cardinals.
  \begin{enumerate}
    \item 
If 
$[\mu, \lambda ]$-compactness is productive
(even, only for members of class $\mathcal C_ \mu$), then
$\mu$ is a $\lambda$-compact cardinal.
\item
If $\mu$ is $\lambda$-compact,
and $\lambda$ is a strong limit singular cardinal
of cofinality $\geq \mu$, then
$[\mu, \lambda ]$-compactness is productive.
  \end{enumerate}
 \end{theorem}

\begin{proof}
(1) First, notice that every 
member of 
$\mathcal C_ \mu$ 
is trivially $[\mu, \lambda ]$-compact.
By assumption, every product of members of $\mathcal C_ \mu$ 
is $[\mu, \lambda ]$-compact. Applying Caicedo \cite[Theorem 3.4]{Ca} 
to $T= \mathcal C_\mu$, we get a
$( \mu, \lambda )$-regular 
 ultrafilter $D$ 
such that every member of $\mathcal C_ \mu $ is $D$-compact; that is, every 
regular cardinal $\mu' < \mu$, with the order topology, is $D$-compact
(notice that \cite{Ca} uses a notation in which the order of the cardinals is reversed).

By \cite[Proposition 1]{topproc}, $D$ is not $(\mu', \mu ')$-regular,
for every regular cardinal $\mu' < \mu$.
By standard arguments (see, e.~g.,  \cite[p. 344]{mru}), if $\kappa$ is
the least infinite cardinal such that a (non principal) ultrafilter $D$ is 
$( \kappa , \kappa )$-regular, then $\kappa$  is a regular cardinal, 
and $D$ is $\kappa$-complete.
In the present case, $\mu \leq \kappa $, hence 
$D$ is $\mu$-complete. Since $D$ 
is also  $( \mu, \lambda )$-regular, then
$\mu$ is $\lambda$-compact.
We have proved (1).

In order to prove (2), we need the following result, whose proof resembles 
\cite[Theorem 1]{hcl}. 

\begin{theorem} \labbel{comehi} 
If $X$ is a $[\mu, \lambda ]$-compact topological space, 
$D$ is a $\mu$-complete ultrafilter over some set $I$, and
$2 ^{|I|} \leq \lambda $, then $X$ is $D$-compact.  
\end{theorem}

\begin{proof} 
Suppose by contradiction that 
$X$ is   $[\mu, \lambda ]$-compact, $D$ is a
$\mu$-complete ultrafilter over $I$, $2 ^{|I|} \leq \lambda  $,
and $X$ is not $D$-compact.
Thus, there is a sequence
$(x_i) _{i \in I} $ of elements of $X$ 
which has no $D$-limit point in $X$.
This means that, for every $x \in X$, there is an open  neighborhood 
$U_x$ of $x$ such that 
$\{ i \in I \mid x_i \not\in U_x  \} \in D$.
For each $x \in X$, choose some $U_x$ as above, and let
$Z_x= \{ i \in I \mid  x_i \not\in U_x  \}$. Thus, $Z_x \in D$. 

For each $Z \in D$, let 
$V_Z = \bigcup \{U_x  \mid x \text{ is such that } Z_x=Z \} $. 
Notice that 
$(V_Z) _{Z \in D} $ is an open cover of $X$, and that
if $i \in Z \in D$, then 
$x_i \not\in V_Z $. 
Since $| D| \leq 2 ^{|I|} \leq \lambda $, then,
by  $[\mu, \lambda ]$-compactness, there is $\mu' < \mu$,
and there is a sequence  
$(Z_ \beta ) _{ \beta \in \mu'} $ of elements of $D$ such that   
$(V _{Z_ \beta }) _{ \beta \in \mu'}$ is a cover of $X$.
Since $D$ is $\mu$-complete,  
$ \bigcap  _{ \beta \in \mu'} Z_ \beta  \in D$,
hence $ \bigcap  _{ \beta \in \mu'} Z_ \beta  \not= \emptyset $.
Choose $i \in \bigcap  _{ \beta \in \mu'} Z_ \beta $. 
Then $x_i \not\in V _{Z_ \beta } $, for every $\beta \in \mu'$,   but this 
contradicts the fact that
$(V _{Z_ \beta }) _{ \beta \in \mu'} $ is a cover of $X$.
\end{proof}  

We are now able to prove Clause (2) in Theorem \ref{thm}. 
So, consider a product $\prod _{j \in J} X_j $
of $[\mu, \lambda ]$-compact topological spaces.
 For every $  \nu < \lambda $,
since $\mu$ is $\lambda$-compact,
there is a $\mu$-complete $(\mu, \nu)$-regular ultrafilter $D$, and
 we can choose $D$ over
$  \nu ^{<\mu}  $
(this is a standard fact, see, e.~g., \cite[Property 1.1(ii)]{mru},
in connection with Form II there).
Letting $ \nu' = \nu ^{<\mu} $, 
we get $ \nu'  < \lambda  $, since $\lambda$ 
is strong limit.

Again since $\lambda$ is strong limit,
$2 ^{ \nu'} < \lambda $,  and,
by Theorem \ref{comehi},  every $X_j$ is $D$-compact, hence 
$\prod _{j \in J} X_j $ is $D$-compact, since
$D$-compactness is productive. 
By \cite[Lemma 3.1]{Ca},
$\prod _{j \in J} X_j $ is 
 $[\mu, \nu ]$-compact, and this holds
for every  $ \nu < \lambda $.
Since 
$\mu \leq \cf \lambda < \lambda $, then 
$\prod _{j \in J} X_j $ is 
 $[\mu, \cf \lambda]$-compact, in particular, 
 $[\cf \lambda, \cf \lambda]$-compact,
and this,
together with  $[\mu, \nu ]$-compactness
for every  $ \nu < \lambda $,
 easily implies  
$[\mu, \lambda ]$-compactness.
 \end{proof}

\begin{corollary} \labbel{nomeas} 
Suppose that  there exists no measurable cardinal. 
If $  \mu \leq \lambda $ are infinite cardinals, and
$[\mu, \lambda ]$-compactness is productive
(even, only for members of   class $\mathcal C_ \mu$), then 
$\mu= \omega $.
\end{corollary}

If $\mu$ is an infinite cardinal, and
$(X_j) _{j \in J} $ are topological spaces,
the \emph{box$_{< \mu} $ product}
$\Box^{<\mu} _{j \in J} X_j $
is a topological space defined on the Cartesian product  
 $\prod _{j \in J} X_j $, and a base of $\Box^{<\mu} _{j \in J} X_j $
 is given
by the family of all products 
$ \prod _{j \in J} O_j $ such that $ O_j $  is open in  $   X_j $,
 for every $j \in J $, and   
$|\{ j \in J \mid O_j \not= X_j \}| < \mu$.
Of course, the  box$_{<  \omega } $ product
is the usual Tychonoff product.

We say that $[\mu, \lambda ]$-compactness is \emph{productive
for $\Box_{<\mu}$  products} (\emph{of members of  some class  $\mathcal K $})
if
every product 
$\Box^{<\mu} _{j \in J} X_j $
of $[\mu, \lambda ]$-compact topological spaces (belonging to $\mathcal K$)
is $[\mu, \lambda ]$-compact.

\begin{corollary} \labbel{extens} 
Suppose that 
$\mu \leq \lambda $ are infinite cardinals, and that
$[\mu, \lambda ]$-compactness is productive
for members of  some class $\mathcal K \supseteq \mathcal C_ \mu$. Then:
  \begin{enumerate}    \item  
$[ \omega , \lambda ]$-compactness,
too, is productive,  for members of  $\mathcal K$.
\item
$[\mu, \lambda ]$-compactness is productive
for $\Box_{<\mu}$  products of members of   $\mathcal K $. 
 \end{enumerate}
\end{corollary}

 \begin{proof}
(1) is a quite easy corollary of Theorem \ref{thm} 
and of Stephenson  and  Vaughan's result.
Notice that a space is $[ \omega , \lambda ]$-compact
if and only if it is both (i) $[\mu, \lambda ]$-compact,
and (ii) $[ \omega , \mu' ]$-compact, for every $\mu' < \mu$.
Since we have proved in Theorem \ref{thm}(1) that,
under the assumptions of the present corollary, 
if $\mu> \omega $, then 
 $\mu$ is (at least)
measurable,  and since measurable cardinals
are inaccessible, 
(ii) above can be equivalently replaced by
(ii)$'$ $[ \omega , \mu' ]$-compact, for every 
singular strong limit
$\mu' < \mu$.
By assumption,
$[\mu, \lambda ]$-compactness is productive for members of $\mathcal K$.
By Stephenson   and  Vaughan's theorem,
$[ \omega , \mu' ]$-compactness is productive
(in the class of all topological spaces, hence also for members of  $\mathcal K$), for every 
singular strong limit
$\mu' < \mu$. Hence $[ \omega , \lambda ]$-compactness,
being the combination of the above properties  productive in $\mathcal K$, 
is productive, too, in $\mathcal K$.

We now prove (2). 
Let  $\mathcal K'$ be the class of those members of $\mathcal K$
which are $[\mu, \lambda ]$-compact.
Applying  \cite[Theorem 3.4]{Ca} 
to $T= \mathcal K'$, we get a
$( \mu, \lambda )$-regular 
 ultrafilter $D$ 
such that every member of $\mathcal K'$ is $D$-compact.
Since $\mathcal K \supseteq \mathcal C_ \mu$,
hence also $\mathcal K' \supseteq \mathcal C_ \mu$, then, 
by the arguments in the proof of Theorem \ref{thm}(1),
we get that $D$ is $\mu$-complete. 
It is easy to see that if $D$ is a $\mu$-complete ultrafilter, then
every $\Box_{<\mu}$  product of $D$-compact spaces is still $D$-compact. 
In the case at hand, we get that every $\Box_{<\mu}$  product of
members of $\mathcal K'$ is $D$-compact, hence 
$[\mu, \lambda ]$-compact, by \cite[Lemma 3.1]{Ca}, 
since $D$ is $( \mu, \lambda )$-regular.
 \end{proof}

As far as we know,  no complete characterization
is known for those cardinals $\lambda$ 
such that  initial $\lambda$-compactness is productive.
A fortiori, we are unable to  give a characterization for those pairs of cardinals 
$\mu$ and $\lambda$ 
such that
$[\mu, \lambda ]$-compactness is productive.
For sure, 
as a consequence of the next proposition,
we know that the assumption
$\cf \lambda \geq \mu$
in Theorem \ref{thm}(2) is necessary 
(unless the existence of strongly compact cardinals is inconsistent).

\begin{proposition} \labbel{counterex}
Suppose that $ \omega \leq \cf \lambda < \mu \leq \lambda $.

Then there is a $[\mu, \lambda ]$-compact $T_5$ topological space $X$
such that $X ^ \kappa  $ fails to be $[ \lambda , \lambda ]$-compact, for
$\kappa =  2 ^{ \lambda  ^{ <\lambda }} $.

In particular, $[\mu, \lambda ]$-compactness is not productive.
 \end{proposition}  

\begin{proof}
Let $X$ be the disjoint union of $\cf\lambda$
and $ \lambda ^+$, both endowed with the order topology.
Observe that $X$ is trivially $[\mu, \lambda ]$-compact,
since $\cf \lambda < \mu$. 
Suppose, by contradiction, that
$X ^ \kappa  $
 is $[ \lambda , \lambda ]$-compact.
Let $I$ be the set of all subsets of $\lambda$
of cardinality $ <\lambda$.
Since $| I |= \lambda ^{< \lambda } $, and 
$| X|= \lambda ^+$, there 
are $ (\lambda ^+) ^{\lambda ^{< \lambda }} =
2 ^{ \lambda  ^{ <\lambda }} $    $I$-indexed sequences of
elements of $X$.
Choose a sequence 
$(x_i) _{i \in I} $ of elements of  $X ^ \kappa  $
in such a way that any $I$-indexed sequence of
elements of $X$ can be obtained as the projection
of $(x_i) _{i \in I} $ onto some factor.
Since we are assuming that 
$X ^ \kappa  $
 is $[ \lambda , \lambda ]$-compact, then,
 by \cite[Lemma 3.3]{Ca},
there is a $(\lambda , \lambda )$-regular ultrafilter $D$ over $I$
such that $(x_i) _{i \in I} $ $D$-converges to some element
of $X ^ \kappa  $. 
Since a sequence in a product $D$-converges if and only if 
every projection converges, we get that 
every $I$-indexed sequence of
elements of $X$ $D$-converges in $X$, that is, $X$ is $D$-compact.
By 
\cite[Corollary 2]{notre} 
(and \cite[Property 1.1(xi)]{mru}),
$D$ is either 
$(\cf\lambda , \cf\lambda )$-regular, or
$(\lambda^+ , \lambda^+ )$-regular.
But then, by \cite[Lemma 3.1]{Ca},
$X$ is either  
$[\cf \lambda , \cf \lambda ]$-compact,
or $[ \lambda^+ , \lambda^+ ]$-compact,
but neither possibility occurs, thus we reached a contradiction, hence
$X ^ \kappa  $ is not $[ \lambda , \lambda ]$-compact.

The last statement follows from the trivial fact that, for $\mu \leq \lambda $,
$[ \mu , \lambda ]$-compactness implies  
$[ \lambda , \lambda ]$-compactness.
 \end{proof} 

For $ \omega <\mu \leq \lambda $, consider the following conditions:
  \begin{enumerate}    
\item[(A)]
 $[\mu, \lambda ]$-compactness is productive.
\item[(B)]
 $\mu$ is $\lambda$-compact,
and 
$[ \omega , \lambda ]$-compactness is productive.
  \end{enumerate}   

It follows from Theorem \ref{thm}(1)
and Corollary \ref{extens}(1) 
that Condition (A) above implies Condition (B).
However, by Proposition \ref{counterex}
and Stephenson  and  Vaughan's Theorem,
if  $\mu$ is $\lambda$-compact, 
and $\lambda$ 
is singular strong limit of cofinality
$<\mu$,
then 
(B) $\Rightarrow $  (A) fails.

The present research is partly motivated by \cite{sssr}.
Let $\mathcal H$ be the set of all
$(\mu, \lambda )$-regular ultrafilters 
over $[ \lambda ] ^{< \mu} $. 
By a remark in \cite{sssr}, if 
$[\mu, \lambda ]$-compactness is productive,
then $\mathcal H$ has a minimum, with respect to 
the Comfort order (see Garc{\'{\i}}a-Ferreira \cite{GF}).

\end{document}